\documentclass{amsart}

\usepackage{amsmath,amssymb}
\usepackage{amsthm}
\usepackage{indentfirst}
\usepackage{graphicx}
\usepackage{hyperref}
\usepackage{xcolor}
\usepackage[margin=1.5in]{geometry}
\usepackage{graphicx,tikz}

\newtheorem*{thm}{Theorem}

\newtheorem{lemma}{Lemma}

\begin{document}

\title[]{Roots of trigonometric polynomials\\ and the Erd\H{o}s-Tur\'{a}n Theorem}

\author[]{Stefan Steinerberger}
\address{Department of Mathematics, Yale University, New Haven, CT 06511, USA}
\email{stefan.steinerberger@yale.edu}

\thanks{S.S. is partially supported by the NSF (DMS-1763179) and the Alfred P. Sloan Foundation.}

\begin{abstract} We prove, informally put, that it is not a coincidence that $\cos{(n \theta)} + 1 \geq 0$ and that the roots of $z^n + 1 =0$ are uniformly distributed in angle -- a version of the statement holds for all trigonometric polynomials with `few' real roots. The
 Erd\H{o}s-Tur\'{a}n theorem states that if $p(z) =\sum_{k=0}^{n}{a_k z^k}$ is suitably
normalized and not too large for $|z|=1$, then its roots
are clustered around $|z| = 1$ and equidistribute in angle at scale $\sim n^{-1/2}$. We establish a connection between the rate of equidistribution of roots in angle and the number of sign changes of the corresponding trigonometric polynomial $q(\theta) = \Re \sum_{k=0}^{n}{a_k e^{i k \theta}}$. If $q(\theta)$ has  $\lesssim n^{\delta}$ roots for some $0 < \delta < 1/2$, then the roots of $p(z)$ do not frequently cluster in angle at scale $\sim n^{-(1-\delta)} \ll n^{-1/2}$.
\end{abstract}

\maketitle

\section{Introduction }
\subsection{Introduction.} 
The Erd\H{o}s-Tur\'{a}n Theorem \cite{erd} shows that polynomials $p:\mathbb{C} \rightarrow \mathbb{C}$
$$ p(z) = \sum_{k=0}^{n}{a_k z^k}$$
that are small on $|z| = 1$ are somewhat rigid: their roots are clustered around the boundary of the unit disk $|z| = 1$ and equidistribute in angle.
Without additional assumptions, there are two obvious counterexamples to this statement: $p(z) = z^n$ (being small on $|z|=1$ but having all roots in the origin) and $p(z) = 1 + \varepsilon z^n$ (having all roots at distance $\sim \varepsilon^{-1/n}$ which is arbitrarily far away from the unit disk). These two examples necessitate
some control on the constant coefficient $a_0$ and the leading coefficient $a_n$. Here and henceforth we shall denote the $n$ roots by $z_1, \dots, z_n$.
Moreover, we will use the quantity $h(p)$ to assign a notion of `size' to a polynomial via
$$ h(p) = \frac{1}{2\pi} \int_{0}^{2\pi}{\log^{+}\left( \frac{|p(e^{i\theta})|}{\sqrt{|a_0|}} \right) d\theta},$$
where $\log^+(x) = \max(0, \log{x})$.
A version of the Erd\H{o}s-Tur\'{a}n Theorem that was recently given by Soundararajan \cite{sound} can then be phrased as follows.

\begin{thm}[Soundararajan \cite{sound}] We have, assuming $|a_n|  = 1$,
$$ \max_{J \subset \mathbb{T}} \left| \frac{\#\left\{1 \leq k \leq n: \arg~z_k \in J\right\}}{n} - \frac{|J|}{2\pi} \right| \leq \frac{8}{\pi} \frac{\sqrt{h(p)}}{\sqrt{n}},$$
where the maximum runs over all intervals $J \subset \mathbb{T}$, where $\mathbb{T}$ is the one-dimensional torus scaled to lentgh $2\pi$ and identified with the boundary of the unit disk.
\end{thm}

Erd\H{o}s and Tur\'{a}n \cite{erd} established the result with constant 16 instead of $8/\pi$. Ganelius \cite{tord} 
gave a proof using a slightly larger quantity than $h(p)$ (this was then improved by Mignotte \cite{mign}). Amoroso \& Mignotte \cite{amor} have shown that the
constant $8/\pi$ cannot be replaced by $\sqrt{2}$. We also refer to two recent variants of the result by Totik \& Varj\'{u} \cite{totik} and Erd\'{e}lyi \cite{erde} and to Blatt \cite{blatt}, Conrey \cite{conrey} and Totik \cite{totik1} for results in a similar spirit.

\subsection{An Observation.} There is an interesting class of trigonometric polynomial that appears naturally in Fourier Analysis and Analytic Number Theory: trigonometric polynomials
that do not vanish. The Fej\'{e}r kernel is a particularly well-known example. If we normalize it to have leading coefficient $|a_n|=1$ it is given by
$$ p_n(\theta) = n+1 + 2\sum_{k=1}^{n}{(n+1-k) \cos{( \theta)}}.$$
Another example is a suitable truncation of the Poisson kernel. Normalized to $|a_n|=1$, it is given by
$$ q_n(\theta) = \rho^{-n} + 2\sum_{k=1}^{n}{\rho^{k-n}\cos{(k \theta)}},$$
where $0 < \rho < 1$ and $n$ has to be sufficiently large depending on $\rho$. A slightly less well-known 1912 example due to W. H. Young \cite{young}, again normalized to $|a_n|=1$, is
$$ r_n(\theta) = n + \sum_{k=1}^{n}{\frac{n}{k} \cos{(k \theta)}}.$$
Nonnegative trigonometric polynomials and their properties have been actively studied for a very long time, we refer to a survey of Dimitrov \cite{dima}, the papers \cite{alzer, and, askey, dima2, rogo, leo} and references therein.
One can take these trigonometric polynomials and interpret them as the real part of a polynomial in the complex plane evaluated on the boundary of the unit disk, this leads to the polynomials
$$ p_n^*(z) = n +1+ 2\sum_{k=1}^{n}{(n+1-k)z^k},~~\qquad~~ q_n^*(z) = \rho^{-n} + 2\sum_{k=1}^{n}{\rho^{k-n} z^k},~~~\qquad~ r_n^*(z) =  n + \sum_{k=1}^{n}{\frac{n}{k} z^k}.$$
We note that the connection between nonnegative trigonometric polynomials and the associated polynomials in the complex plane has been studied before. Most notably, there is a rich interaction in terms of univalence: in 1910 Fej\'{e}r conjectured that
$$ \sum_{k=1}^{n}{ \frac{\sin{(k \theta)}}{k}} \geq 0 \qquad \mbox{for}~0 \leq \theta \leq \pi.$$
This was proven by Jackson \cite{jackson} in 1911 and Gr\"onwall \cite{gronwall} in 1912. In his 1915 PhD thesis, J. W. Alexander showed \cite{alexander} that
$$ \sum_{k=1}^{n}{\frac{z^k}{k}} \qquad \mbox{is univalent on the unit disk}$$
and these two results are clearly at least related in spirit. The connection was made definitive by Dieudonne \cite{dieu} in 1931 when he proved that
$$ \sum_{k=1}^{n}{a_k z^k} \qquad \mbox{is univalent if and only if} \qquad \sum_{k=1}^{n}{a_k z^{k-1} \frac{\sin{(k \theta)}}{\sin{(\theta)}}} \neq 0$$
for all $z \in \mathbb{D}$, where $\mathbb{D}$ denotes the unit disk, and all $0 \leq \theta \leq \pi.$ We refer to the survey by Gluchoff \& Hartmann \cite{gluchoff} for more details of the story.
We ask, partially motivated by Fig. 1, a different question: what can be said about the zeroes of these polynomials? Fig. 1 suggests an interesting answer: they are extremely regularly spaced.
 We were interested in whether this was coincidence and this motivated the paper. We note that it is easy to construct nonnegative trigonometric polynomials whose associated polynomials have roots that are not quite as regular (see the third example in Fig. 1) but even then we observe regularity for `most' angles.
\begin{figure}[h!]
\begin{minipage}[l]{.32\textwidth}
\centering
\includegraphics[width = 4cm]{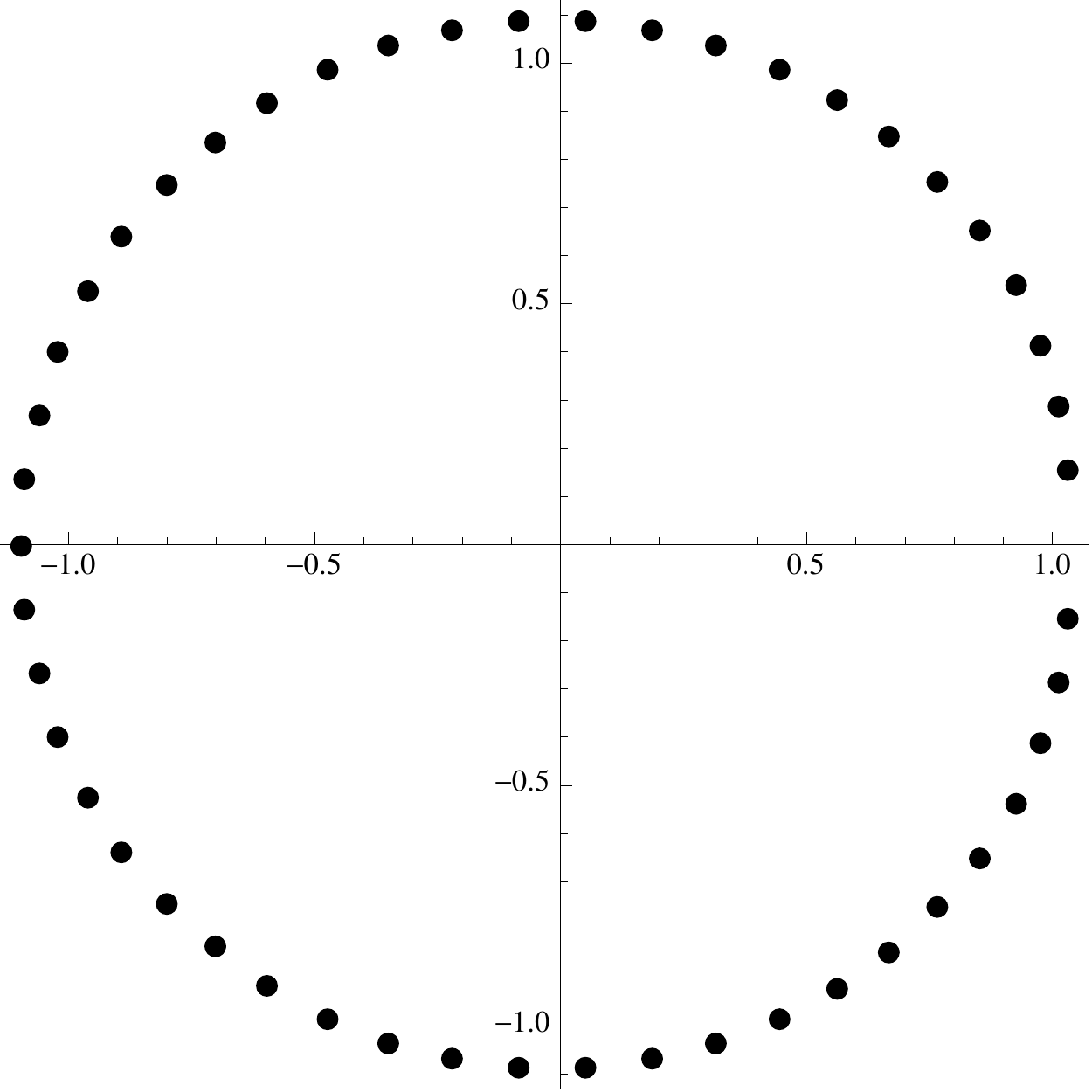} 
\end{minipage} 
\begin{minipage}[c]{.32\textwidth}
\includegraphics[width = 4cm]{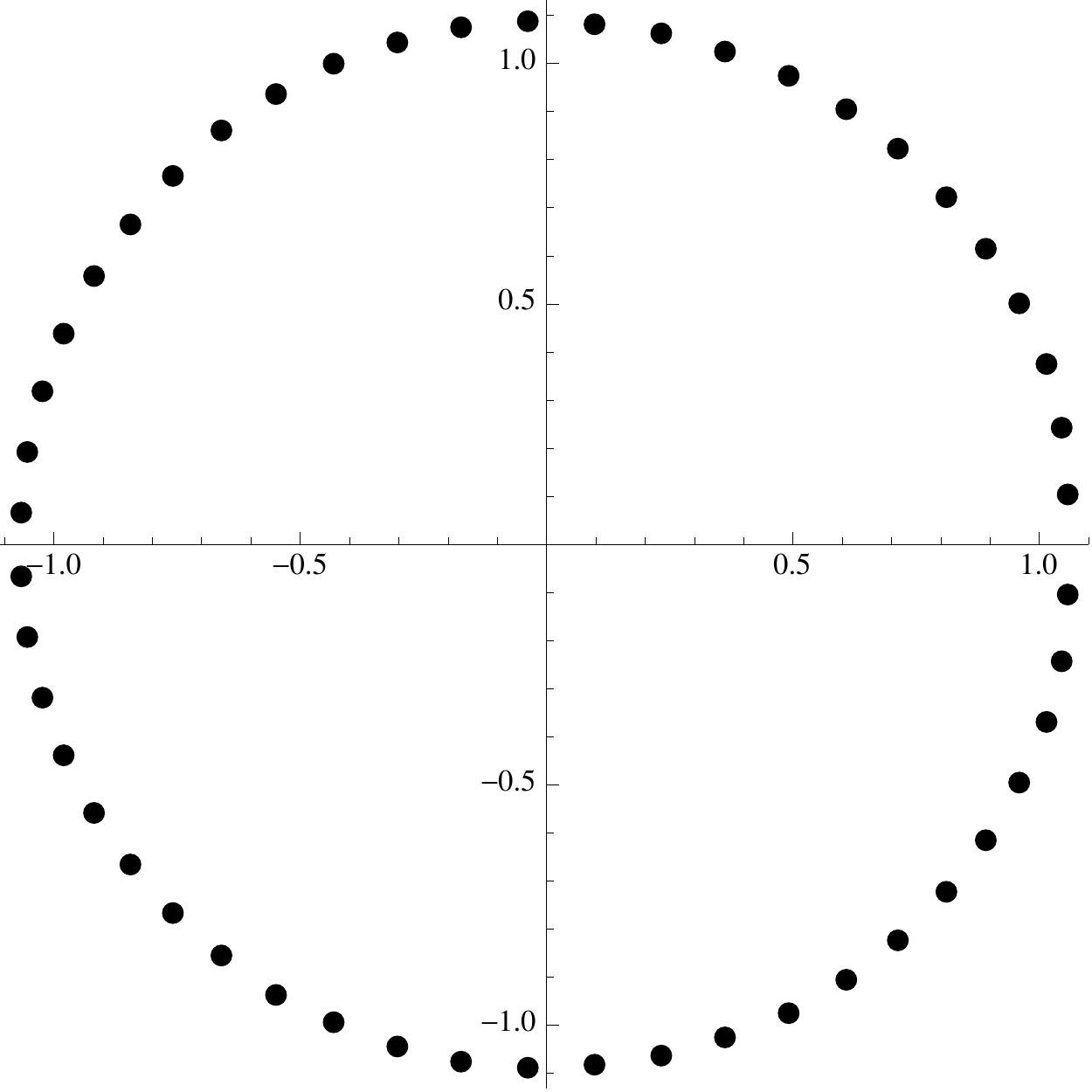} 
\end{minipage} 
\begin{minipage}[r]{.32\textwidth}
\includegraphics[width = 4cm]{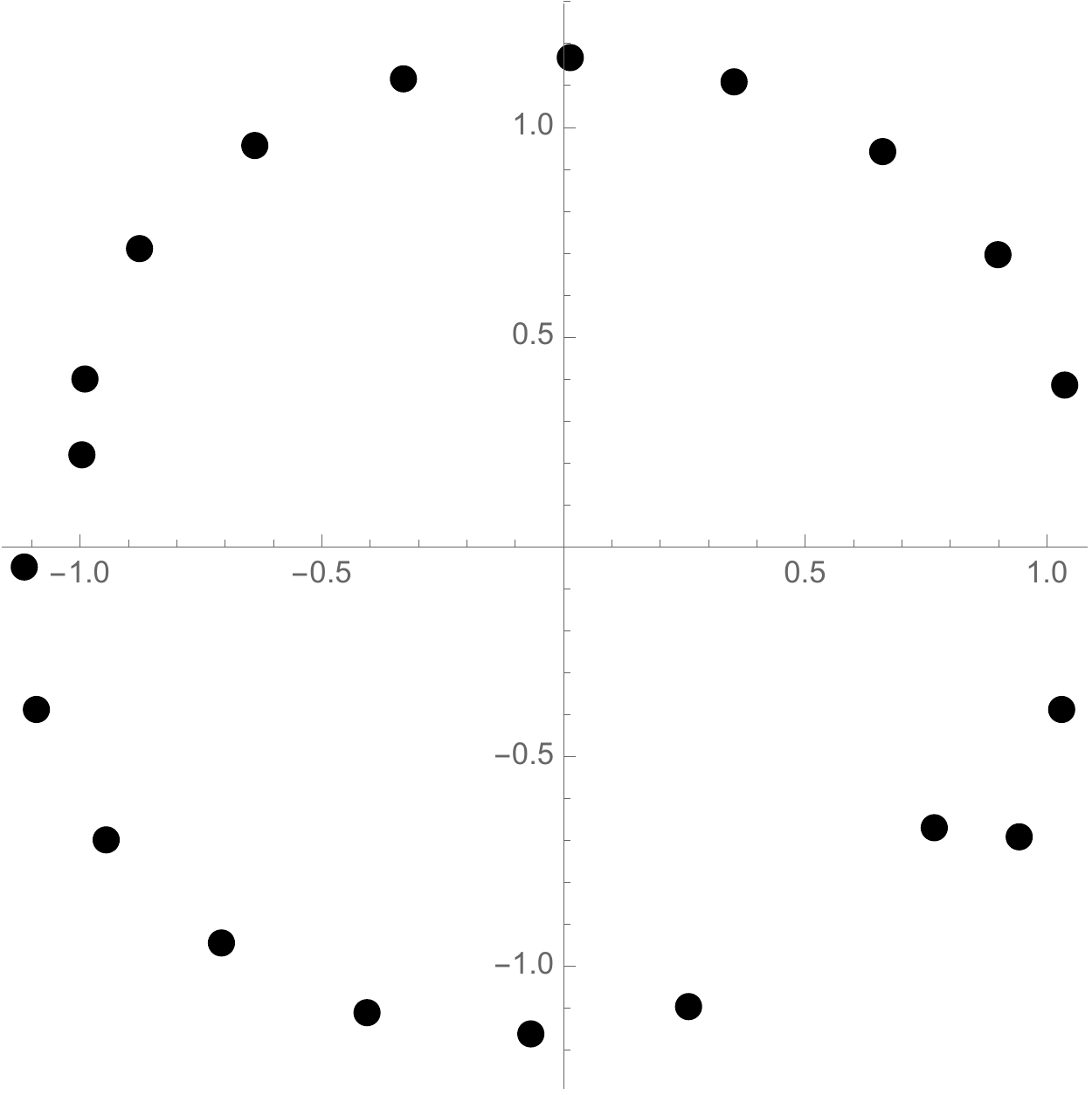} 
\end{minipage} 
\caption{Roots of $p_{50}^*(z)$ (left),  $r_{50}^*(z)$ (center) and $p_{20}^*(z)+p_{20}^*(e^{i}z)$ (right).}
\end{figure}

We need to be mindful of the Erd\H{o}s-Tur\'{a}n theorem: roots of polynomials already equidistribute in angle without any further assumptions (apart from a normalization) at scale $\sim n^{-1/2}$. We will establish, for polynomials whose corresponding trigonometric polynomial is nonnegative or has few roots, a regularity statement at lower scales (indeed, all the way down to $n^{-1}$). We do not claim that this is necessarily the only regularity statement one could obtain and hope that this paper can inspire some subsequent work.

\subsection{Result.} We can now state our main result. For any interval $J \subset \mathbb{T} \cong [0,2\pi]$ we will say that $J$ contains too many roots if the number of roots that lie in the corresponding sector is five times larger than we would expect for a typical interval of length $|J|$, i.e.
$$ \# \left\{1 \leq k \leq n: \arg z_k \in J \right\} \geq 5 \frac{|J|}{2\pi} n.$$
We now state the result: intervals containing too many roots are rare.
\begin{thm} Assume the polynomial $p(z) = \sum_{k=0}^{n}{a_k z^k}$ is normalized to $|a_n| = 1$ and 
$$ \Re \sum_{k=0}^{n}{a_k e^{i k\theta}} \qquad \mbox{has}~X~\mbox{real roots on}~[0, 2\pi].$$
Then, for $0 \leq \alpha \leq 1$, the maximum number $I$ of disjoint intervals of length $n^{-\alpha}$ for which
$$ \# \left\{1 \leq k \leq n: \arg z_k \in J \right\} \geq 5\frac{ |J|}{2\pi} n,$$
  can be bounded by
$$ I \lesssim  n^{\alpha - 1}X + n^{2\alpha -1}   \int_{\partial \mathbb{D}}{ \log{|p(e^{it})|} dt}.$$
\end{thm}

There are several ways of interpreting the statement.
\begin{enumerate}
\item An interesting special case is $\alpha = 1/2$ corresponding to the scaling of the classical Erd\H{o}s-Tur\'{a}n theorem. 
Here it says that trigonometric polynomials with $\lesssim \sqrt{n}$ roots are better behaved than the Erd\H{o}s-Tur\'an theorem predicts except for possibly a very small number of intervals (the integral is $\sim \log{n}$ for $p_n$ and $r_n$ in \S 1.2).\\
\item Let us consider trigonometric polynomials with very few roots, $X \lesssim n^{\varepsilon}$, and $\alpha = 1-\varepsilon$. This corresponds to very short intervals $n^{-\alpha} = n^{\varepsilon} n^{-1}$. As it turns out, most of these intervals behave as expected unless the logarithmic integral is large. In particular, as long as the logarithmic integral is smaller than $n^{\varepsilon/2}$, we see that at most $n^{1-(3/2)\varepsilon}$ intervals are `dense' in the sense of containing five times as many roots as expected. There are $\sim n^{1-\varepsilon}$ intervals in total, we conclude that a typical interval does not have any clustering. Positive trigonometric polynomials give rise to polynomials whose roots are extremely regularly structured at small scale; to the best of our knowledge, this property of trigonometric polynomials is new.\\
\item The statement is empty for $X=n$. However, it can be used as a contrapositive statement: if many (say, 1\%) of intervals at scale $n^{-\alpha}$ exceed expectation by a factor of 5, then either the polynomial is extremely large or the associated trigonometric polynomial has to have $\sim n$ roots.\\
\item There is a symmetry in the statement: we can always multiply the polynomial by a fixed complex number of modulus 1: this shows that we can replace the assumption  
$$ \Re \sum_{k=0}^{n}{a_k e^{i k\theta}} \qquad \mbox{having}~X~\mbox{real roots on}~[0, 2\pi]$$
by
$$  \Im  \sum_{k=0}^{n}{a_k e^{i k\theta}}  \qquad \mbox{having}~X~\mbox{real roots on}~[0, 2\pi]$$
or, more generally, for any $0 < x < 2\pi$, the equation
$$ \arg \sum_{k=0}^{n}{a_k e^{ikt}} = x \qquad \mbox{having}~X~\mbox{solutions on}~[0, 2\pi].$$
\item We also remark that the classical Erd\H{o}s-Tur\'an theorem does not apply since we only normalize $|a_n|=1$ but not $a_0$. There is, however, a refinement due to Erd\'{e}lyi \cite{erde} who showed that under solely that normalization there still is a regularity statement for those roots that are outside the unit disk. In our setting, $\Re p(e^{it})$ having few roots, the argument principle implies that most of the roots will be outside.\\
\item We remark that our result refines the Erd\H{o}s-Turan theorem as soon as the number of roots of the associated trigonometric polynomials has $\lesssim \sqrt{n}$ roots. In the generic setting, we expect such a trigonometric polynomial to have $\sim n$ roots. It is an interesting problem whether there are improvements of our results as soon as the associated trigonometric polynomial has $\lesssim n^{1-\varepsilon}$ roots.
\end{enumerate}

\section{Proof}
We start with some elementary considerations. We will make use of the argument $\arg(e^{it}-z_k)$ which is only defined up to multiples of $2\pi$. To bypass the usual difficulty of defining it, we will instead work with its derivative that has a nice global definition.
\subsection{A Lemma.} 
\begin{lemma} Let $|z| > 1$ be given by $z = r \cos{\theta} + i r \sin{\theta}$ (for $r>1$). Then
$$ \frac{d}{dt} \arg(e^{it} - z) = \frac{1-r \cos{(t - \theta)}}{1 - 2r \cos{(t - \theta)} + r^2}.$$
\end{lemma}
\begin{proof} We first assume w.l.o.g. that $z \in \mathbb{R}$ and $z < 0$. In that case, we have
$$ \arg(e^{it} - z) = \arctan{\left( \frac{\sin{t}}{\cos{t} + z}\right)}.$$
A computation shows that 
$$ \frac{d}{dt} \arctan{\left( \frac{\sin{t}}{\cos{t} + z}\right) }= \frac{1 + z \cos{t}}{1 + 2z \cos{t} + z^2}.$$
The general case then follows from rotational invariance.
\end{proof}
As expected, for $r > 1$,
 $$\int_{0}^{2\pi}{\frac{1-r \cos{(t - \theta)}}{1 - 2r \cos{(t - \theta)} + r^2}dt} = 0.$$
Moreover, we will use the inequality, valid for any $1 < r \leq 1.1$ that
\begin{equation} \label{drop}
\int_{\theta-(r-1)}^{\theta+(r-1)}{\frac{1-r \cos{(t - \theta)}}{1 - 2r \cos{(t - \theta)} + r^2}dt} \leq -\frac{3}{2}.
\end{equation}
Here, the constant $-3/2$ could be replaced by $-1.52$ but this is of no further importance for the rest of the argument. The proof of the inequality
proceeds by rewriting the integrand as a derivative of the argument which turns the integral into an elementary trigonometric problem. The curvature of the circle dictates that the integral is monotonically increasing in $r$ which shows that it suffices to evaluate the integral for $r=1.1$.
We also note an elementary inequality, valid for all $r>1$,
\begin{equation} \label{drop2}
  \frac{1-r \cos{(t - \theta)}}{1 - 2r \cos{(t - \theta)} + r^2} \leq \frac{1}{1+r} \leq \frac{1}{2}
  \end{equation}
and observe that for all $r>1$ and all intervals $J \subset [0,2\pi]$
$$\int_{J}^{}{\frac{1-r \cos{(t - \theta)}}{1 - 2r \cos{(t - \theta)} + r^2}dt} \leq \pi.$$
For points inside the disk, this is slightly different since the argument is monotonically increasing and thus, for all $|z|<1$ and $J \subset  [0,2\pi]$
$$0 \leq \int_{J} \frac{d}{dt} \arg(e^{it} - z) dt \leq  \int_{0}^{2\pi} \frac{d}{dt} \arg(e^{it} - z) dt = 2\pi.$$

\subsection{Proof of the Theorem.}

\begin{proof} We write
$$ p(z) = e^{i \theta}\prod_{k=1}^{n}{(z-z_k)} \qquad \mbox{for some}~\theta \in [0, 2\pi].$$
The constant $\theta$ will not be of any importance for subsequent arguments (it corresponds to a rotation of the complex plane but the statement is rotation-invariant), thus we can set $\theta = 0$ without loss of generality. A simple consequence is
$$\arg p(e^{it}) = \sum_{k=1}^{n}{ \arg( e^{it} - z_k)}.$$

The argument principle implies that $|z_k| > 1$ for all but at most $X$ roots: by the argument principle, every root inside the unit disk corresponds to winding around the origin once which necessitates crossing the $y-$axis twice -- this corresponds to the real part vanishing and happens at most $X$ times in total.
 We argue that (this is a form of Jensen's theorem but also follows from straightforward computation)
\begin{equation} \label{jensen}
 \int_{\partial \mathbb{D}}{ \log{|p(e^{it})|} dt} = \sum_{k=1}^{n}{   \int_{\partial \mathbb{D}}{ \log{|e^{it}-z_k|} dt} } = 2\pi \sum_{|z_k| > 1}^{n}{ \log{|z_k|}}.
\end{equation}

We now assume that $(I_{\alpha})_{\alpha}$ is a collection of disjoint intervals of length $n^{-\alpha}$ all of which have the property that the number of roots whose argument is in that interval is a factor 5 larger than expected
$$ \# \left\{1 \leq k \leq n: \arg z_k \in I_j \right\} \geq 5 \frac{ n^{1-\alpha}}{2\pi}.$$
Let us now assume that $J$ is such an interval. We will study the size of
$$ \int_{J}{ \frac{d}{dt} \arg p(e^{it}) dt}.$$
By the fundamental Theorem of Calculus, this is merely the total change of the argument over the interval. If this quantity is very large (say $>2\pi k$ for some $k \in \mathbb{N}$) or very small (say $<-2\pi k$), then $p(e^{it})$ must necessarily wind around the origin $k$ times creating $2k$ roots of $\Re p(e^{it})$ in the process. For trigonometric polynomials with no or few roots, we therefore expect that this integral is close to 0 for most intervals.
We introduce the set of roots that are in the corresponding cone, $A = \left\{z_i: \arg z_i \in J\right\}$, and the set of remaining roots $B = \left\{z_i: \arg z_i \notin J\right\}$. Any root that is in the origin, $z_k=0$, has an undefined argument and we define those to not be in $A$ and to be in $B$. We can split the integral as
$$ \int_{J}{ \frac{d}{dt} \arg p(e^{it}) dt} =\int_{J}{ \sum_{z_k \in A}{   \frac{d}{dt} \arg (e^{it} - z_k) dt}} + \int_{J}{ \sum_{z_k \in B}{   \frac{d}{dt} \arg (e^{it} - z_k) dt}}.$$

We subdivide $B$ into the set $B_1 = \left\{z_k \in B: |z_k| \geq 1\right\}$ and $B_2 = B \setminus B_1$. Equation \eqref{drop2} implies
\begin{equation} \label{bound}
 \int_{J}{ \sum_{z_k \in B_1}{   \frac{d}{dt} \arg (e^{it} - z_k) dt}}  \leq \frac{1}{2} |B_1| |J| \leq \frac{n}{2} n^{-\alpha}.
 \end{equation}
We decompose the set $A$ into three sets 
\begin{align*}
A_1 &= \left\{z_k \in A: |z_k| \geq 1 + n^{-\alpha} \right\}, \\
A_2 &= \left\{z_k \in A: 1 <  |z_k| < 1 + n^{-\alpha} \right\}, \\
A_3 &= \left\{z_k \in A: 0 < |z_k| \leq 1 \right\}.
\end{align*}
All three sets are somewhat problematic: roots in $A_1$ are far away from the unit disk and increase the size of the polynomial by Jensen's formula. Roots in $A_3$ necessarily create roots of $\Re p(e^{it})$ by the argument principle. This leaves $A_2$: here, we have that roots contribute a great deal to the argument integral. This effect can be counterbalanced by contributions from other roots that are globally distributed and yield a contribution of the opposite sign. 
\begin{figure}[h!]
\begin{center}
\begin{tikzpicture}[scale=1.2]
\draw [thick] (0,0) circle (2cm);
\draw [ultra thick,domain=0:45] plot ({2*cos(\x)}, {2*sin(\x)});
\draw [dashed] (0,0) -- (3.5,0);
\draw [dashed] (0,0) -- (2.4,2.4);
\node at (1,0.5) {\LARGE $A_3$};
\draw [dashed,domain=0:45] plot ({2.8*cos(\x)}, {2.8*sin(\x)});
\node at (2.1,1.1) {\LARGE $A_2$};
\node at (3,1.5) {\LARGE $A_1$};
\node at (-2.2,1.2) {\LARGE $B_1$};
\node at (-2.2,-1.2) {\LARGE $B_1$};
\node at (1.2,2.2) {\LARGE $B_1$};
\node at (3.2,-1.2) {\LARGE $B_1$};
\node at (-0.4,0) {\LARGE $B_2$};
\end{tikzpicture}
\end{center}
\caption{The regions into which we subdivide.}
\end{figure}
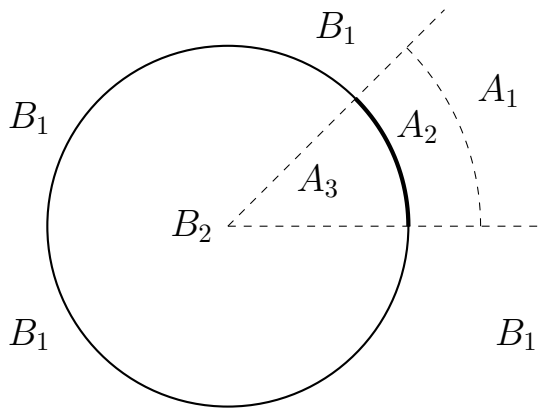
We now make this procise
and start by noting that, by assumption, 
\begin{equation} \label{alot}
|A_1| + |A_2| + |A_3|\geq  5\frac{n^{1-\alpha}}{2\pi}.
\end{equation}
However, if $A_1$ is big, then these roots contribute a lot to the size of $\log{|p(z)|}$ since
\begin{equation} \label{gain}
\sum_{z_k \in A_1}{\log{|z_k|}} \geq \log{(1 + n^{-\alpha})} |A_1| \geq \frac{n^{-\alpha}}{2} |A_1|.\end{equation}
Conversely, if $|A_2|$ is big, we have from \eqref{drop} and $z_k \in A_2$ that
$$  \int_{J}{ \sum_{z_k \in A_2}{   \frac{d}{dt} \arg (e^{it} - z_k) dt}}  \leq -\frac{3|A_2|}{4}.$$
For every $|z_k| < 1$ and every $ 0 \leq t \leq 2\pi$
$$ \frac{d}{dt} \arg (e^{it} - z_k) \geq 0 \qquad \mbox{and} \qquad \int_{J}  \frac{d}{dt} \arg (e^{it} - z_k) dt \leq 2\pi$$ 
and this results in the trivial esimate
$$  \int_{J}{ \sum_{z_k \in A_3}{   \frac{d}{dt} \arg (e^{it} - z_k) dt}}  \leq  2 \pi |A_3|.$$
We bound the number of roots of the associated trigonometric polynomial over the arc by
\begin{align*}
 \# \left\{t \in J: \Re p(e^{it}) = 0\right\} \geq \left| \int_{J}{ \frac{d}{dt} \arg p(e^{it}) dt} \right| 
 \end{align*}
 and resolve the integral by using a reverse triangle inequality. More precisely, omitting the actual functions (which are merely the arguments over the corresponding terms), we argue, using the just derived estimates and \eqref{bound},
 \begin{align*}
   \left| \int_{J}{ \frac{d}{dt} \arg p(e^{it}) dt} \right| &=  \left| \int_{J} \sum_{z_k \in A_1} + \int_{J} \sum_{z_k \in A_2} + \int_{J} \sum_{z_k \in A_3} + \int_{J} \sum_{z_k \in B_1} + \int_{J} \sum_{z_k \in B_2} \right| \\
   &\geq \left| \int_{J} \sum_{z_k \in A_2} \right| - \left| \int_{J} \sum_{z_k \in A_1} \right|  - \left| \int_{J} \sum_{z_k \in A_3} \right|  - \left| \int_{J} \sum_{z_k \in B_1} \right|  - \left| \int_{J} \sum_{z_k \in B_2} \right| \\
&\geq   \frac{3 |A_2|}{4} - \pi |A_1| - 2\pi |A_3|  - \frac{n^{1-\alpha}}{2}  - \left| \int_{J} \sum_{z_k \in B_2} \right|   
 \end{align*}
Altogether, this shows that
\begin{equation} \label{nice}
  \# \left\{t \in J: \Re p(e^{it}) = 0 \right\} + \pi |A_1| + 2\pi |A_3| + \left| \int_{J} \sum_{z_k \in B_2} \right|   \geq \frac{3}{4}|A_2| - \frac{n^{1-\alpha}}{2}.
  \end{equation}
  We note that this inequality only conveys information when $|A_2| > 2 n^{1-\alpha}/3$.
The next idea is to introduce the term
$$ n^{-\alpha} X+ \sum_{|z_k| > 1}{\log{|z_k|}},$$
where $X$ is the total number of roots of $\Re p(e^{it})$,
as quantity and argue that this quantity, when summed over all such intervals $J$ from the collection of intervals, is large. The key ingredient is \eqref{alot} implying that at least one of the sets $A_1, A_2, A_3$ is large (where large means $\gtrsim n^{1-\alpha}$). We distinguish three cases: (a) Suppose $|A_1| \gtrsim n^{1-\alpha}$. Then, by \eqref{gain}, the term goes up at least by $\gtrsim n^{1-2\alpha}$. (b) If $|A_3| \gtrsim n^{1-\alpha}$, then the global number of roots of $p(z)$ that are contained inside the unit disk goes up by $\gtrsim n^{1-\alpha}$ and thus, by the argument principle, our quantity of interest increases by at least $\gtrsim n^{1-2\alpha}$. The remaining case (c) is that both these quantities, $|A_1|$ and $|A_3|$ are small, say $\leq 10^{-100} n^{1-\alpha}$. However, using \eqref{alot} and \eqref{nice}, we see that then necessarily
$$  \# \left\{t \in J: \Re p(e^{it}) = 0\right\} + \left| \int_{J} \sum_{z_k \in B_2} \right| \gtrsim n^{1-\alpha}.$$
Suppose now there are a total of $k$ disjoint intervals of length $n^{-\alpha}$ having too many roots that decompose into the three cases via $k = k_1 + k_2 + k_3$ where $k_1$ counts the number of intervals resulting in case (a), $k_2$ the number resulting in case (b) and $k_3$ the number resulting in case (c). We have already shown that
\begin{equation} \label{half}
 n^{-\alpha} X+ \sum_{|z_k| > 1}{\log{|z_k|}} \gtrsim (k_1 + k_2) n^{1-2\alpha}
 \end{equation}
and if $k_1 + k_2 \sim k$, then we have established the desired result. However, perhaps the first two cases hardly appear and the third case is the most frequent and $k_3 \sim k$. Summing over these cases, we obtain
\begin{align*}
 k_3 \cdot n^{1-\alpha} \lesssim \sum_{J} \left(  \# \left\{t \in J: \Re p(e^{it}) = 0\right\} + \left| \int_{J} \sum_{z_k \in B_2(J)} \right| \right),
 \end{align*}
 where the notation indicates that $B_2(J)$ is always defined in relation to the interval $J$. One sum can be bounded easily via
 $$ \sum_{J} \# \left\{t \in J: \Re p(e^{it} = 0) \right\} \lesssim X.$$
 We note that roots in $B_2$ are all inside the unit disk and thus the argument is always positive and, conversely, any root inside the unit disk always contributes a positive argument. This allows us to bound
\begin{align*}
 \sum_{J} \left|  \int_{J}{ \sum_{z_k \in B_2(J)}{   \frac{d}{dt} \arg (e^{it} - z_k) dt}}\right| &=
  \sum_{J}   \int_{J}{ \sum_{z_k \in B_2(J)}{   \frac{d}{dt} \arg (e^{it} - z_k) dt}} \\
  &\leq   \sum_{J}   \int_{J}{ \sum_{|z_k| \leq 1}{   \frac{d}{dt} \arg (e^{it} - z_k) dt}} \\
  &\leq       \int_{0}^{2\pi}{ \sum_{|z_k| \leq 1}{   \frac{d}{dt} \arg (e^{it} - z_k) dt}} \\
&\leq 2\pi \# \left\{z_k: |z_k| \leq 1\right\} \leq 2\pi X.
  \end{align*}
 Altogether, this shows that
 \begin{align*}
 k_3 \cdot n^{1-\alpha}  \lesssim  X + n^{\alpha} \sum_{|z_k| > 1}{\log{|z_k|}}  \end{align*}
which, when multiplying with $n^{\alpha-1}$ and adding to \eqref{half} implies
$$ k = k_1+k_2+k_3 \lesssim X n^{\alpha-1} + n^{2\alpha-1} \sum_{|z_k| > 1}{\log{|z_k|}}  $$
which, combined with Jensen's formula \eqref{jensen}, implies the result. 
  \end{proof}

\end{document}